\documentclass[a4paper, 11pt]{amsart}
\usepackage[
  margin=1.9cm,
  includefoot,
  footskip=30pt,
]{geometry}
\usepackage{amsmath,amssymb,amscd}
\usepackage{amsfonts}
\usepackage[english]{babel}
\usepackage[leqno]{amsmath}
\usepackage{amssymb,amsthm}
\usepackage{mathrsfs} 
\usepackage{amscd}
\usepackage{enumerate}
\usepackage{epsfig}
\usepackage{relsize}
\usepackage{layout}
\usepackage[usenames,dvipsnames]{xcolor}
\usepackage[backref=page]{hyperref}
\usepackage{tikz}
\hypersetup{
 colorlinks,
 citecolor=Green,
 linkcolor=Red,
 urlcolor=Blue}
\usepackage[matrix,arrow,tips,curve]{xy}
\input{xy}
\xyoption{all}
\definecolor{darkgreen}{rgb}{0.0, 0.7, 0.0}
\definecolor{purple}{rgb}{0.5, 0.0, 0.5}
\definecolor{red}{rgb}{0.8, 0.2, 0.0}

\newtheorem{thm}{Theorem}[section]
\newtheorem{bthm}{Theorem}

\newtheorem{lemma}[thm]{Lemma}

\newtheorem{prop}[thm]{Proposition}

\numberwithin{equation}{section}
\theoremstyle{definition}
\newtheorem{defi}[thm]{Definition}

\newtheorem{notation}[thm]{Notation}

\theoremstyle{remark}
\newtheorem{remark}[thm]{Remark}
\newtheorem{remarks}[thm]{Remarks}

\newcommand{\Pic}{\operatorname{Pic}}

\def \Im{{\rm Im}}

\def \PP{\mathbb{P}}
\def \ZZ{\mathbb{Z}}

\def \F{\mathcal F}

\def \L{\mathcal L}
\def \E{\mathcal E}
\def \G{\mathcal G}
\def \H{\mathcal H}
\def \U{\mathcal U}
\def\O{\mathcal O}

\def\M0{\mathcal M^0}

\begin{document}

\title[On the positivity of the first Chern class of an Ulrich vector bundle]{On the positivity of the first Chern class of an Ulrich vector bundle}

\author[A.F. Lopez]{Angelo Felice Lopez*}

\thanks{* Research partially supported by  PRIN ``Geometria delle variet\`a algebriche'' and GNSAGA-INdAM}

\address{\hskip -.43cm Dipartimento di Matematica e Fisica, Universit\`a di Roma
Tre, Largo San Leonardo Murialdo 1, 00146, Roma, Italy. e-mail {\tt lopez@mat.uniroma3.it}}

\thanks{{\it Mathematics Subject Classification} : Primary 14J60. Secondary 14J40, 14J30.}

\dedicatory{\normalsize \dag \ Dedicated to all medical workers and people that sacrificed their lives during the Covid-19 pandemic} 

\begin{abstract} 
We study the positivity of the first Chern class of a rank $r$ Ulrich vector bundle $\E$ on a smooth $n$-dimensional variety $X \subseteq \PP^N$. We prove that $c_1(\E)$ is very positive on every subvariety not contained in the union of lines in $X$. In particular if $X$ is not covered by lines we have that $\E$ is big and $c_1(\E)^n \ge r^n$. Moreover we classify rank $r$ Ulrich vector bundles $\E$ with $c_1(\E)^2=0$ on surfaces and with $c_1(\E)^2=0$ or $c_1(\E)^3=0$ on threefolds (with some exceptions).
\end{abstract}

\maketitle

\section{Introduction}
\label{intro}

Let $X \subseteq \PP^N$ be a smooth projective variety. A classical method to study the geometry of $X$ is to produce interesting vector bundles with some kind of positivity on $X$. These in turn give rise to meaningful subvarieties shedding light on the geometrical properties of $X$.

Among such bundles, in recent years a distinguished role has been achieved by Ulrich vector bundles, that is vector bundles $\E$ such that $H^i(\E(-p))=0$ for all $i \ge 0$ and $1 \le p \le \dim X$.

The study of Ulrich vector bundles touches and is intimately related with several areas of commutative algebra and algebraic geometry, such as determinantal representation, Chow forms, Boij-S\"oderberg theory and so on (see for example \cite{es, b1, fl} and references therein), even though their basic existence question is still open (see for example \cite{hub, b1, f, c1, c2, cohu, lo} in the case of surfaces). 

The starting observation of this paper is that an Ulrich vector bundle is globally generated, whence it must positive in some sense. The question is: how much is it positive? We will investigate in this paper the positivity of its first Chern class.

Let $\E$ be a rank $r$ Ulrich vector bundle on $X$. It is a simple observation (see Lemma \ref{c1=0}) that
$c_1(\E) = 0$ if and only if $(X,\O_X(1), \E)=(\PP^n, \O_{\PP^n}(1), \O_{\PP^n}^{\oplus r})$. 

Aside from this case, we have that $c_1(\E)$ is globally generated and non-zero, so the question becomes: How positive is $c_1(\E)$?

The first answer, surprisingly related to the projective geometry of $X$, is given in Theorem \ref{main} below. It shows that whenever $X \subset \PP^N$ is not covered by lines then $c_1(\E)$ is quite positive and that $\E$ is big. More precisely we have

\begin{bthm}
\label{main}

\hskip 3cm

Let $X \subseteq \PP^N$ be a smooth irreducible variety of dimension $n \ge 1$ and degree $d$. Let $t \ge 1$ and let $\E$ be a rank $r$ Ulrich vector bundle for $(X, \O_X(t))$. Then $c_1(\E)$ has the following positivity property: 

For every $x \in X$ and for every subvariety $Z \subseteq X$ of dimension $k \ge 1$ passing through $x$, we have that 
\begin{equation}
\label{ineq}
c_1(\E)^k \cdot Z \ge r^k {\rm mult}_x(Z)
\end{equation}
holds if either $t \ge 2$ or $t=1$ and $x$ does not belong to a line contained in $X$. 

In particular $c_1(\E)$ is ample if either $t \ge 2$ or $t=1$ and there is no line contained in $X$.

Moreover assume that either $t \ge 2$ or $t=1$ and $X$ is not covered by lines. Then $\E$ is big,
\begin{equation}
\label{ineq'}
c_1(\E)^n \ge r^n
\end{equation}
and, if $r \ge 2$,
\begin{equation}
\label{ineq''}
c_1(\E)^n \ge r(d-1).
\end{equation}
\end{bthm}

Other positivity properties of Chern and dual Segre classes of $\E$ actually hold, see Remark \ref{altrapos}.

It is easy to see that the results of Theorem \ref{main} do not hold on varieties covered by lines. In particular \eqref{ineq'} is sharp, in the sense that for every $n \ge 2$ there are many $n$-dimensional varieties such that 
\eqref{ineq'} does not hold: for example products $\PP^1 \times B$, scrolls or quadrics and complete intersections in any of them (see Lemmas \ref{pic} and \ref{exa} and Remark \ref{sez}).

On the other hand, varieties covered by lines are classified in dimension $n \le 3$ and one could try to similarly classify their Ulrich vector bundle with $c_1(\E)$ not positive. We will concentrate on the case when $c_1(\E)^n = 0$.

We first recall the following

\begin{defi}
\label{lin2}
Let $X \subset \PP^N$ be a smooth irreducible variety of dimension $n \ge 1$. We say that $(X, \O_X(1))$ is a {\it linear $\PP^{n-1}$-bundle over a smooth curve $B$} if there are a rank $n$ vector bundle $\F$ and a line bundle $L$ on $B$ such that $X \cong \PP(\F)$ and $\O_X(1) \cong \O_{\PP(\F)}(1) \otimes \pi^*L$, where $\pi : \PP(\F) \to B$ is the projection. 
\end{defi}

With this in mind, we now classify all rank $r$ Ulrich vector bundles on surfaces with $c_1(\E)^2=0$. 

\begin{bthm}
\label{main2}

\hskip 3cm

Let $S  \subseteq \PP^N$ be a smooth irreducible surface. Let $\E$ be a rank $r$ Ulrich vector bundle on $S$. 

Then $c_1(\E)^2=0$ if and only if $(S,\O_S(1), \E)$ is one of the following:
\begin{itemize}
\item [(i)] $(\PP^2, \O_{\PP^2}(1), \O_{\PP^2}^{\oplus r})$;
\item [(ii)] $(\PP(\F), \O_{\PP(\F)}(1) \otimes \pi^*L, \pi^*(\G(L+ \det \F)))$, a linear $\PP^1$-bundle over a smooth curve $B$, with $\G$ a rank $r$ vector bundle on $B$ such that $H^q(\G(-L))=0$ for $q \ge 0$.
\end{itemize} 
\end{bthm}

Note that the vector bundle $\G$ above is not necessarily Ulrich. See also Remark \ref{div} for other observations.

In the case of threefolds, let us recall the following

\begin{defi}
\label{lin3}
We say that $(X, \O_X(1))$ is a {\it linear quadric fibration over a smooth curve $B$}, if there is a morphism $\pi : X \to B$ such that each fiber $Q$ is isomorphic to a quadric and $\O_X(1)_{|Q} \cong \O_Q(1)$.
\end{defi}

Again we have a classification, though with some exceptions, of rank $r$ Ulrich vector bundles on threefolds first when $c_1(\E)^2=0$ and then when only $c_1(\E)^3=0$. 

\begin{bthm}
\label{main3}

\hskip 3cm

Let $X \subseteq \PP^N$ be a smooth irreducible threefold. Let $\E$ be a rank $r$ Ulrich vector bundle on $X$. 

Assume that $(X,\O_X(1))$ is neither a linear quadric fibration nor a linear $\PP^1$-bundle over a smooth surface with $\Delta(\F)=0$. 

Then $c_1(\E)^2=0$ if and only if $(X,\O_X(1), \E)$ is one of the following:
\begin{itemize}
\item [(i)] $(\PP^3, \O_{\PP^3}(1), \O_{\PP^3}^{\oplus r})$;
\item [(ii)] $(\PP(\F), \O_{\PP(\F)}(1) \otimes \pi^*L, \pi^*(\G(2L+ \det \F)))$, a linear $\PP^2$-bundle over a smooth curve $B$, with  $\G$ a rank $r$ vector bundle on $B$ such that $H^q(\G(-L))=0$ for $q \ge 0$; 
\item [(iii)] $(\PP(\F), \O_{\PP(\F)}(1) \otimes \pi^*L, \pi^*(\G(L+ \det \F)))$, a linear $\PP^1$-bundle over a smooth surface $B$, with $\G$ a rank $r$ vector bundle on $B$ such that $H^q(\G(-L)) = H^q(\G(-2L) \otimes \F^*)=0$ for $q \ge 0$ and  $c_1(\G(L+ \det \F))^2=0$.
\end{itemize}
\end{bthm}

We do not know if the case (iii) actually occurs.

\begin{bthm}
\label{main4}

\hskip 3cm

Let $X \subseteq \PP^N$ be a smooth irreducible threefold. Let $\E$ be a rank $r$ Ulrich vector bundle on $X$. 

Assume that $(X, \O_X(1))$ is neither a linear quadric fibration, nor a linear $\PP^1$-bundle over a smooth surface, nor a linear $\PP^2$-bundle over a smooth curve with $\F$ normalized and $\deg \F \ge 0$, unless $(X, \O_X(1))$ is also a Del Pezzo threefold.

Then $c_1(\E)^3=0, c_1(\E)^2 \ne 0$ if and only if $(X, \O_X(1), \E)$ is one of the following:
\begin{itemize}
\item [(i)] $(\PP^1 \times \PP^1 \times \PP^1, \O_{\PP^1}(1) \boxtimes \O_{\PP^1}(1) \boxtimes \O_{\PP^1}(1),(\O_{\PP^1} \boxtimes \O_{\PP^1}(1) \boxtimes \O_{\PP^1}(2))^{\oplus s} \oplus (\O_{\PP^1} \boxtimes \O_{\PP^1}(2) \boxtimes \O_{\PP^1}(1))^{\oplus (r-s)})$ for $0 \le s \le r$;
\item [(ii)] $(\PP(T_{\PP^2}), \O_{\PP(T_{\PP^2})}(1),\pi^*(\O_{\PP^2}(2))^{\oplus r})$, where $\pi$ is one of the two $\PP^1$-bundle structures on $\PP(T_{\PP^2})$.
\end{itemize}
\end{bthm}

In the case of a linear quadric fibration different from $\PP^1 \times \PP^1 \times \PP^1$ we do not know if there are Ulrich vector bundles $\E$ with $c_1(\E)^3=0$, but we suspect there aren't (see Proposition \ref{qf} and the discussion just before).

Throughout the whole paper we work over the complex numbers.

\section{Ulrich vector bundles with $c_1=0$}
\label{c1nullo}

It is easy, and probably well-known, to classify Ulrich vector bundles with $c_1=0$. We give a proof for a lack of reference and for self-containedness.
\begin{lemma}
\label{c1=0}
Let $X \subseteq \PP^N$ be a smooth irreducible variety of dimension $n \ge 1$. Let $\E$ be a rank $r$ Ulrich vector bundle on $X$. Then $c_1(\E) = 0$ if and only if $(X,\O_X(1), \E)=(\PP^n, \O_{\PP^n}(1), \O_{\PP^n}^{\oplus r})$.
\end{lemma}
\begin{proof} 
If $(X,\O_X(1))=(\PP^n,\O_{\PP^n}(1))$ it is well-known \cite[Prop.~2.1]{es}, \cite[Thm.~2.3]{b1}, that $\E = \O_{\PP^n}^{\oplus r}$ and of course $c_1(\E) = 0$.

Viceversa suppose that $\E$ is a rank $r$ Ulrich vector bundle with $c_1(\E) = 0$ on $X$. It follows by \cite[(3.4)]{b1} that for any smooth $Y \in |\O_X(1)|$ we have that $\E_{|Y}$ is an Ulrich vector bundle  for $Y \subseteq \PP^{N-1}$ and of course $c_1(\E_{|Y}) = 0$. Therefore if $C = Y_1 \cap \ldots \cap Y_{n-1}$ is a smooth curve section of $X$, with $Y_i \in |\O_X(1)|$ for $1 \le i \le n-1$, then $\E_{|C}$ is an Ulrich vector bundle for $(C,\O_C(1))$ and $\deg (\E_{|C}) = 0$. By \cite[Prop.~2.3]{ch} we get that 
$$0 = \deg (\E_{|C}) = r (\deg C + g(C)-1).$$
Thus $\deg C = 1 - g(C)$, hence $\deg C = 1$ and then $\deg X = 1$. Therefore $(X,\O_X(1))=(\PP^n,\O_{\PP^n}(1))$. It follows by \cite[Prop.~2.1]{es} (or \cite[Thm.~2.3]{b1}) that $\E = \O_{\PP^n}^{\oplus r}$.
\end{proof} 

\section{Varieties of Picard rank $1$}
\label{rk1}

It has been observed by several authors that on varieties of Picard rank $1$ more information can be given on Ulrich vector bundles. We concentrate here on the first Chern class.

\begin{notation}
We will denote by $Q_n$ an $n$-dimensional smooth quadric in $\PP^{n+1}$.
\end{notation}

We have

\begin{lemma}
\label{pic}
Let $X \subseteq \PP^N$ be a smooth irreducible variety of dimension $n \ge 2$ such that $\Pic X \cong \ZZ$. Let $\E$ be a rank $r$ Ulrich vector bundle on $X$. Let  $H$ be the hyperplane divisor. Then 
$$c_1(\E) = \frac{r}{2}(K_X+(n+1)H).$$
In particular if $(X,\O_X(1))\ne (\PP^n,\O_{\PP^n}(1))$ we have
$$c_1(\E)^n = \begin{cases} \frac{r^n}{2^{n-1}}< r^n & {\rm if} \ (X,\O_X(1))=(Q_n,\O_{Q_n}(1)), n \ge 3 \\ 
\frac{r^n}{2^n}(K_X+(n+1)H)^n > r^n & {\rm if} \ (X,\O_X(1)) \ne (Q_n,\O_{Q_n}(1)) \end{cases}.$$
\end{lemma}
\begin{proof} 
Let $H_1, \ldots, H_{n-2}$ be general in $|H|$ and let $S = X \cap H_1 \cap \ldots \cap H_{n-2} \subseteq \PP^{N-n+2}$. It follows by \cite[(3.4)]{b1} that  $\E_{|S}$ is a rank $r$ Ulrich vector bundle on $S$ and \cite[Prop.~2.2(3)]{c1} gives
\begin{equation}
\label{sup}
c_1(\E_{|S}) \cdot H_{|S} = \frac{r}{2}(3H_{|S}^2 + K_S \cdot H_{|S}).
\end{equation}
Let $A$ be an ample generator of $\Pic X$, so that we have $H = hA, K_X = kA$ and $c_1(\E) = aA$, for some $h, k , a \in \ZZ$. Replacing in \eqref{sup} we get
$$ah A_{|S}^2 =  \frac{rh}{2}(k + (n+1)h) A_{|S}^2$$
from which we deduce that $a =  \frac{r}{2}(k + (n+1)h)$ and therefore 
$$c_1(\E) = \frac{r}{2}(K_X+(n+1)H).$$ 
Moreover if $(X,\O_X(1)) \ne (\PP^n,\O_{\PP^n}(1)), (Q_n, \O_{Q_n}(1))$, it follows by \cite[Prop.~7.2.2, 7.2.3 and 7.2.4]{bs2}) that $D:= K_X+(n-1)H$ is nef, hence
$$c_1(\E)^n = \frac{r^n}{2^n}(D+2H)^n \ge r^nH^n > r^n.$$
If $(X,\O_X(1))=(Q_n,\O_{Q_n}(1))$, then $n \ge 3$, since $\Pic X \cong \ZZ$. In the latter case we have that
$K_X=-nH$ and therefore $c_1(\E)^n = \frac{r^n}{2^{n-1}}< r^n$.
\end{proof}

On the other hand, when the Picard rank is not $1$, one can achieve that $c_1(\E)^n =0$ even though $(X,\O_X(1))\ne (\PP^n,\O_{\PP^n}(1))$. Several examples are given in the next section.

\section{Ulrich vector bundles on scrolls}

A big source of examples are given by scrolls. These are treated, to mention a few, in \cite[Prop.~4.1(ii)]{b1} and in \cite{ho, flcpl, ahmpl}.

We study here the case when $c_1(\E)^n =0$.

\begin{lemma}
\label{exa}
Let $B$ be a smooth irreducible variety of dimension $b \ge 1$ and let $L$ be a divisor on $B$. Let $n$ be an integer such that $n \ge b+1$ and let $\F$ be a rank $n-b+1$ vector bundle on $B$. Let $X = \PP(\F)$ with tautological line bundle $\xi$ and projection $\pi: X \to B$. Let $H = \xi + \pi^*L$ and assume that $H$ is very ample. Let $\G$ be a rank $r$ vector bundle on $B$ and let 
$$\E = \pi^*(\G((n-b)L+ \det \F)).$$  
Then $\E$ is an Ulrich vector bundle for $(X, H)$ if and only if 
\begin{equation}
\label{ul}
H^q(\G(-(k+1)L)\otimes S^k \F^*)=0 \ \hbox{for} \ q \ge 0 \ \hbox{and} \ 0 \le k \le b-1.
\end{equation}
In particular $\E$ is an Ulrich vector bundle for $(X, H)$ if
\begin{itemize}
\item [(i)] $b=1$ and $H^q(\G(-L))=0$ for $q \ge 0$ (in particular if $L$ is very ample and $\G$ is an Ulrich vector bundle for $(B, L))$, or
\item [(ii)] $b \ge 2, \F = \O_B^{\oplus (n-b+1)}$ and $H^q(\G(-kL))=0$ for $q \ge 0$ and $1 \le k \le b$  (in particular if $L$ is very ample and $\G$ is an Ulrich vector bundle for $(B, L))$.
\end{itemize}
Moreover $c_1(\E)^s = 0$ for every integer $s \ge b+1$.
\end{lemma}
\begin{proof} 
Let $p$ be an integer such that $1 \le p \le n$. Let $j \ge 0$. By \cite[Ex.~III.8.4]{ha} we see that
\begin{equation}
\label{rj}
R^j \pi_*\O_{\PP(\F)}(-p) =  \begin{cases} 0 & {\rm if} \ j \ne n-b \ \hbox{or} \ j=n-b \ \hbox{and} \ 1 \le p \le n-b \\ 
S^{p-n+b-1} \F^*(-\det \F) & {\rm if} \  j=n-b \ \hbox{and} \ n-b+1 \le p \le n \end{cases}.
\end{equation}
Now
\begin{equation}
\label{rj'}
R^j \pi_*\E(-pH) \cong R^j \pi_*\E(-p\xi)(-pL) \cong \G((n-b-p)L + \det \F) \otimes R^j \pi_*\O_{\PP(\F)}(-p)
\end{equation}
whence \eqref{rj} and the Leray spectral sequence give that 
\begin{equation}
\label{ul1}
H^i(\E(-pH))=0 \ \hbox{for} \ i \le n-b-1 \ \hbox{and for} \ i=n-b \ \hbox{and} \ p \le n-b.
\end{equation}
On the other hand, for $n-b \le i \le n$ and $n-b+1 \le p \le n$, again using \eqref{rj}, \eqref{rj'} and the Leray spectral sequence, we have
$$H^i(\E(-pH)) \cong H^{i-n+b}(\G((n-b-p)L)\otimes S^{p-n+b-1} \F^*)$$
that is the equivalence with \eqref{ul} holds.

Now in cases (i) or (ii) one easyly sees that \eqref{ul} is equivalent to assert that $H^q(\G(-kL))=0$ for $q \ge 0$ and $1 \le k \le b$. Finally since $c_1(\E) = \pi^*(c_1(\G) + r(n-b)L + r\det \F)$ it is clear that $c_1(\E)^s = 0$ for every integer $s \ge b+1$.
\end{proof}

\begin{remarks} (about Lemma \ref{exa}) \hskip 4cm
\label{div}

\begin{itemize}
\item [(1)] The condition that $H = \xi + \pi^*L$ is very ample can be easily achieved, for example if $\F$ is very ample and $L$ is globally generated.
\item [(2)] In order to have that $\E$ is an Ulrich vector bundle, the conditions in (i) or (ii) do not require $L$ to be any positive. For example choose $\F$ very ample, $L=0$, $b=1$ and $\G$ is a general line bundle on $B$ of degree $g(B)-1$ or $b \ge 2$, $B=\PP^b$ and $\G=\O_{\PP^b}(-1)$. 
\item [(3)] In case (ii) we have that $X = \PP^{n-b} \times B$ and $H = \pi_1^*\O_{ \PP^{n-b}}(1) + \pi_2^*L$. This shows that in products $\PP^{n-b} \times B$ there are always Ulrich vector bundles as long as there are on $B$.
\end{itemize}
\end{remarks}

\begin{remark}
\label{sez}
Let $Y$ be any smooth variety of dimension $n+k$ carrying an Ulrich vector bundle $\E$ for $(Y,H)$ such that  $c_1(\E)^s = 0$ for some integer $s$. Let $X$ be a smooth complete intersection of $H_1,\ldots, H_k \in |H|$. Then $\E_{|X}$ is an Ulrich vector bundle for $(X,H_{|X})$ by \cite[(3.4)]{b1} and again $c_1(\E_{|X})^s = 0$.
\end{remark}

Using this one can produce several examples of varieties with Ulrich vector bundles $\E$ with $c_1(\E)^s = 0$. For example one can take complete intersections in products. One special case will be studied in section \ref{dp}.

\section{A basic lemma}
\label{basic}

We record a useful result that will be applied several times.

\begin{lemma}
\label{pullback}
Let $\pi : X \to Y$ be a morphism between two smooth projective varieties such that $\pi_*\O_X = \O_Y$. Let $\E$ be a rank $r$ globally generated vector bundle on $X$. If there is a line bundle $M$ on $Y$ such that $\det \E = \pi^*M$, then there is a rank $r$ vector bundle $\H$ on $Y$ such that $\E = \pi^*\H$.
\end{lemma}
\begin{proof}
Consider the map
$$\lambda_{\E} : \Lambda^r H^0(\E) \to H^0(\det \E).$$
The linear system $|\Im \lambda_{\E}|$ is base-point free: In fact, given any point $x \in X$, there are $s_1,\ldots s_r \in H^0(\E)$ such that $s_1(x),\ldots s_r(x)$ are linearly independent in $\E_x$. Hence the section $\lambda_{\E}(s_1 \wedge \ldots \wedge s_r) \in \Im \lambda_{\E}$ does not vanish at $x$. Now $\E$ is globally generated, whence it defines a morphism 
$$\Phi_{\E}: X \to G(H^0(\E), r)$$ 
such that 
$$\E = \Phi_{\E}^*\U,$$ 
where $\U$ is the tautological bundle on $G(H^0(\E),r)$. Moreover we have a commutative diagram (see for example \cite[\S 3]{mu})
$$\xymatrix{X \ar[d]^{\varphi_{|\Im \lambda_{\E}|}} \ar[r]^{\hskip -.9cm \Phi_{\E}} & G(H^0(\E), r) \ar@{^{(}->}[d]^P \\ \PP \Im \lambda_{\E}^* \ar@{^{(}->}[r] & \PP \Lambda^r H^0(\E)^*}$$
where $P$ is the Pl\"ucker embedding. Then $Z: = \Phi_{\E}(X) = \varphi_{|\Im \lambda_{\E}|}(X)$ and we deduce that there is a morphism $\varphi_{|\Im \lambda_{\E}|} : X \to Z$ and a rank $r$ vector bundle $\G:= \U_{|Z}$ on $Z$ such that $\E \cong \varphi_{|\Im \lambda_{\E}|}^*\G$. 

If $W = \varphi_{\det \E}(X)$, using the fact that $\Im \lambda_{\E} \subseteq H^0(\det \E)$, we also have a commutative diagram
$$\xymatrix{X \ar[dr]_{\varphi_{|\Im \lambda_{\E}|}} \ar[r]^{\varphi_{\det \E}} & W \ar[d]^{f} \\ & Z}$$
and we find that $\E \cong \varphi_{\det \E}^*(f^*\G)$. 

Since $\det \E = \pi^*M$ is globally generated and $\pi_*\O_X = \O_Y$, we have that $M$ is globally generated and there is a commutative diagram
$$\xymatrix{X \ar[dr]_{\varphi_{\det \E}} \ar[r]^{\pi} & Y \ar[d]^{\varphi_M} \\ & W}.$$
Therefore $\E \cong \pi^*\H$ where $\H$ is the rank $r$ vector bundle $\varphi_M^*(f^*\G))$.
\end{proof}

The above lemma allows to classify Ulrich vector bundles $\E$ on a linear $\PP^{n-1}$-bundle over a curve when $c_1(\E)^s = 0$ for $2 \le s \le n-1$.

\begin{lemma}
\label{cs}
Let $B$ be a smooth irreducible curve and let $L$ be a divisor on $B$. Let $n$ be an integer such that $n \ge 3$ and let $\F$ be a rank $n$ vector bundle on $B$. Let $X = \PP(\F)$ with tautological line bundle $\xi$ and projection $\pi: X \to B$. Let $H = \xi + \pi^*L$ and assume that $H$ is very ample. Let $\E$ be an Ulrich vector bundle for $(X, H)$ with $c_1(\E)^s = 0$ for some $s$ with $2 \le s \le n-1$. Then there exists a rank $r$ vector bundle $\G$ on $B$ such that $H^q(\G(-L))=0$ for $q \ge 0$ and
$$\E = \pi^*(\G((n-1)L+ \det \F)).$$  
\end{lemma}
\begin{proof} 
We can write $\det \E = a \xi + \pi^*M$ for some $a \in \ZZ$ and some line bundle $M$ on $B$, so that there is $b \in \ZZ$ such that $c_1(\E) = a \xi + b f$, where $f$ is a fiber. Then 
$$0 = c_1(\E)^s = a^s \xi^s + sa^{s-1} b \xi^{s-1} \cdot f.$$
Since $\xi^s$ and $\xi^{s-1} \cdot f$ are numerically independent, we deduce that $a=0$ and then $\det \E = \pi^*M$. Now Lemma \ref{pullback} implies that there is a rank $r$ vector bundle $\H$ on $B$ such that $\E \cong \pi^*\H$. Setting $\G = \H((1-n)L-\det \F)$, we find that $\E \cong \pi^*(\G((n-1)L+ \det \F))$ and Lemma \ref{exa}(i) implies that $H^q(\G(-L))=0$ for $q \ge 0$. 
\end{proof}

\section{Ulrich vector bundles on Del Pezzo threefolds covered by lines}
\label{dp}

The goal of this section is to study Ulrich vector bundles $\E$ with $c_1(\E)^3=0$ on $C \times \PP^1 \times \PP^1$, where $C$ is a curve and on $\PP(T_{\PP^b})$. In particular when $C = \PP^1$ and $b=2$ we get the two examples \cite[Thm.~1.4]{lp}, \cite[Introduction]{fu} of Del Pezzo threefolds covered by lines. 

We start with the simplest linear quadric fibration. The lemma will show that we cannot hope to construct Ulrich vector bundles $\E$ with $c_1(\E)^3=0$ by pulling back from the base. For this reason we suspect that, in the case of linear quadric fibrations different from $\PP^1 \times \PP^1 \times \PP^1$, there aren't any such bundles.
 
\begin{prop}
\label{qf}
Let $C$ be a smooth irreducible curve. Let $X = C \times \PP^1 \times \PP^1$ and let $H = \L \boxtimes \O_{\PP^1}(1) \boxtimes \O_{\PP^1}(1)$, where $\L$ is a very ample line bundle on $C$. Then the only rank $r$ Ulrich vector bundles $\E$ for $(X,H)$ such that $c_1(\E)^3=0$ are:
\begin{itemize}
\item [(i)] $\pi^*(\G(L))$, where $\pi: X \to C \times \PP^1$ is one of the two projections, $L = \L \boxtimes \O_{\PP^1}(1)$ and $\G$ is a rank $r$ Ulrich vector bundle for $(C \times \PP^1, L)$; 
\item [(ii)] $(\O_{\PP^1} \boxtimes \O_{\PP^1}(1) \boxtimes \O_{\PP^1}(2))^{\oplus s} \oplus (\O_{\PP^1} \boxtimes \O_{\PP^1}(2) \boxtimes \O_{\PP^1}(1))^{\oplus (r-s)}$ for $0 \le s \le r$, when $C = \PP^1$ and $\L = \O_{\PP^1}(1)$.
\end{itemize}
In particular there are no rank $r$ Ulrich vector bundles $\E$ for $(X,H)$ such that $c_1(\E)^2=0$.
\end{prop}
\begin{proof}
Let $\pi_1 : X \to C, \pi_i : X \to \PP^1, i = 2, 3$ and $\pi_{ij}=(\pi_i,\pi_j)$ for $1 \le i < j \le 3$ be the projections. For $i=2, 3$ set $H_i = \pi_i^*(\O_{\PP^1}(1))$  and, given any line bundle $N$ on $C$, set $N_1 = \pi_1^*N$. By \cite[Exer.~III.12.6]{ha} we know that there is line bundle $M$ on $C$ and there are $a, b \in \ZZ$ such that $\det \E = M_1+ aH_2 + bH_3$.

Then $(\det \E)^3=0$ gives $6ab \deg M = 0$, whence either $\deg M=0$ or $a=0$ or $b=0$. Since $\E$ is globally generated we get that $M$ is globally generated and $a \ge 0, b \ge 0$. Therefore, if $\deg M = 0$ we have that $M = \O_C$. We deduce that there are three possibilities for $\det \E$: 
$$\pi_{12}^*(M \boxtimes \O_{\PP^1}(a)), \pi_{13}^*(M \boxtimes \O_{\PP^1}(b)) \ \hbox{and} \ \pi_{23}^*(\O_{\PP^1}(a) \boxtimes \O_{\PP^1}(b)).$$ 
By Lemma \ref{pullback} we deduce that there is a rank $r$ vector bundle $\H$ on $C \times \PP^1$ (respectively on $\PP^1\times \PP^1$) such that $\E \cong \pi_{1i}^*\H$ for $i=2, 3$ (respectively $\E \cong \pi_{23}^*\H$). 

Suppose that $\E \cong \pi_{1i}^*\H$ for $i=2, 3$. We have that $X \cong \PP(\O_{C \times \PP^1}^{\oplus 2})$ and $\xi \cong \pi_3^*(\O_{\PP^1}(1))$, so that $H = \xi + \pi_{1i}^*(\L \boxtimes \O_{\PP^1}(1))$. We can then apply Lemma \ref{exa}(ii) and deduce that $\E \cong \pi_{1i}^*(\G(L))$, where $L = \L \boxtimes \O_{\PP^1}(1)$ and $\G:= \H(-L)$ is a rank $r$ Ulrich vector bundle for $(C \times \PP^1, L)$.

Suppose instead that $\E \cong \pi_{23}^*\H$. Set, for convenience of notation, $Q = \PP^1 \times \PP^1$ and $\O_Q(1) = \O_{\PP^1}(1) \boxtimes \O_{\PP^1}(1)$. We have $H = \pi_1^*\L + \pi_{23}^*(\O_Q(1))$ and, for every $1 \le i, p \le 3$, the K\"unneth formula gives
\begin{equation}
\label{ku}
0 = H^i(X, \E(-pH)) = \bigoplus\limits_{\alpha+\beta=i} H^\alpha(C, -p\L) \otimes H^\beta(Q, \H(-p))= H^1(C,-p\L) \otimes H^{i-1}(Q,\H(-p)).
\end{equation}
If $(C, \L) \ne (\PP^1, \O_{\PP^1}(1))$ we have $H^1(C,-p\L) \ne 0$, and therefore we get that 
$$H^{i-1}(Q,\H(-p))=0 \ \hbox{for every} \ 1 \le i, p \le 3.$$ 
But then
$$0 = \chi(\H(-p)) = r + \frac{1}{2}c_1(\H)^2 - c_2(\H) + (1-p) c_1(\H) \cdot \O_Q(1) + rp^2-2rp.$$
Now 
$$0 =  \chi(\H(-2)) - \chi(\H(-1)) = - c_1(\H) \cdot \O_Q(1) + r$$
and
$$0 =  \chi(\H(-3)) - \chi(\H(-2)) = - c_1(\H) \cdot \O_Q(1) + 3r$$
thus giving the contradiction $2r = 0$.

Finally if $(C, \L) = (\PP^1, \O_{\PP^1}(1))$ then $H^1(C,-\L) = 0$ and $H^1(C,-p\L) \ne 0$ for $p=2, 3$, whence \eqref{ku} gives that $H^{i-1}(Q,\H(-p))=0$ for $1 \le i \le 3$ and $p=2, 3$. But then $\G := \H(-1)$ is a rank $r$ Ulrich vector bundle on $Q$. By \cite[Rmk.~2.5(4)]{bgs} (see also \cite[Rmk.~2.6]{b1}, \cite[Exa.~3.2]{ahmpl}) the only indecomposable Ulrich bundles for $(Q, \O_Q(1))$ are $\O_{\PP^1} \boxtimes \O_{\PP^1}(1)$ and $\O_{\PP^1}(1) \boxtimes \O_{\PP^1}$. This gives that 
$$\E \cong (\O_{\PP^1} \boxtimes \O_{\PP^1}(1) \boxtimes \O_{\PP^1}(2))^{\oplus s} \oplus (\O_{\PP^1} \boxtimes \O_{\PP^1}(2) \boxtimes \O_{\PP^1}(1))^{\oplus (r-s)}$$ 
for $0 \le s \le r$.

Now in case (i) we have that $c_1(\E) = \pi^*(c_1(\G(L)))$ whence $c_1(\E)^2 = \pi^*((c_1(\G)+rL)^2)$. But
$\G$, being Ulrich, is globally generated, whence $c_1(\G)+rL$ is ample. Therefore $c_1(\E)^2 \ne 0$.

In case (ii) we have that $c_1(\E) = (2r-s)H_2+(r+s)H_3$ and therefore 
$$c_1(\E)^2 = 2(2r-s)(r+s)H_2H_3 \ne 0.$$ 
\end{proof}

An interesting example, related to Remark \ref{sez}, is given by the projectivised tangent bundle of projective spaces. In the case of $\PP^2$ this is a Del Pezzo threefold covered by lines. We recall (see \cite[Thm.~A]{sa}) that $\PP(T_{\PP^b})$ has two $\PP^{b-1}$-bundle structures, $\pi: \PP(T_{\PP^b}) \to \PP^b$ and $\pi' : \PP(T_{\PP^b}) \to \PP^b$, given by the fact that $\PP(T_{\PP^b})$ is a divisor of type $(1,1)$ on $\PP^b \times \PP^b$.

\begin{prop}
\label{exa1}
Let $b \ge 2$ and let $X = \PP(T_{\PP^b})$ with tautological line bundle $\xi$ and projection $\pi: X \to \PP^b$. Then $\xi$ is very ample and $\pi^*(\O_{\PP^b}(b))$ is an Ulrich line bundle for $(X, \xi)$.
In particular $c_1(\pi^*(\O_{\PP^b}(b)))^s = 0$ for $b+1 \le s \le 2b-1$.

Moreover the only rank $r$ Ulrich vector bundles $\E$ for $(X, \xi)$ with $c_1(\E)^{2b-1}=0$ are $\pi^*(\O_{\PP^b}(b))^{\oplus r}$ and $(\pi')^*(\O_{\PP^b}(b))^{\oplus r}$. In particular there are no Ulrich vector bundles $\E$ for $(X, \xi)$ with $c_1(\E)^s=0$ for $1 \le s \le b$.
\end{prop}
\begin{proof} 
The Euler sequence
$$0 \to \O_{\PP^b} \to \O_{\PP^b}(1)^{\oplus (b+1)} \to T_{\PP^b} \to 0$$
gives an inclusion $\PP(T_{\PP^b}) \subset \PP(\O_{\PP^b}(1)^{\oplus (b+1)}) \cong \PP^b \times \PP^b$ and $\xi$ is the restriction of the tautological line bundle of $\PP(\O_{\PP^b}(1)^{\oplus (b+1)})$, that is the restriction of $\O_{\PP^b}(1) \boxtimes \O_{\PP^b}(1)$. Hence $\xi$ is very ample.

Set $n = 2b-1, \F = T_{\PP^b}, L = 0$ and $\G = \O_{\PP^b}(-1)$ in Lemma \ref{exa}. Then 
$$\pi^*(\O_{\PP^b}(b)) = \pi^*(\G((n-b)L+ \det \F))$$ 
and by Lemma \ref{exa}(ii) all we need are the vanishings
\begin{equation}
\label{schn}
H^q(S^k \Omega^1_{\PP^b}(-1))=0 \ \hbox{for} \ q \ge 0 \ \hbox{and} \ 0 \le k \le b-1.
\end{equation}
By \cite[\S 1]{sc} there is an exact sequence
$$0 \to S^k \Omega^1_{\PP^b}(-1) \to \O_{\PP^b}(-k-1)^{\oplus \binom{k+b}{b}} \to \O_{\PP^b}(-k)^{\oplus \binom{k+b-1}{k-1}} \to 0$$
and \eqref{schn} follows immediately. This proves that $\pi^*(\O_{\PP^b}(b))$ is an Ulrich line bundle for $(X, \xi)$.

Now set $R = \pi^*H$, where $H$ is a hyperplane in $\PP^b$ and let $f$ be a fiber of $\pi$. 

We claim that
\begin{equation}
\label{monom}
\xi^{b-1+j} \cdot R^{b-j} = \binom{b+j}{j} \ \hbox{for every} \ 0 \le j \le b.
\end{equation}
We prove this by induction on $j$. If $j=0$ we have
$$\xi^{b-1} \cdot R^b = \xi^{b-1} \cdot f = 1.$$
If $j \ge 1$ recall that we have the relation 
$$\sum\limits_{i=0}^b (-1)^i \pi^*c_i(T_{\PP^b}) \xi^{b-i}=0$$
and since $c_i(T_{\PP^b}) = \binom{b+1}{i} H^i$, this gives
$$\xi^b = \sum\limits_{i=1}^b (-1)^{i+1} \binom{b+1}{i} \xi^{b-i}\cdot R^i$$
and therefore
$$\xi^{b-1+j} \cdot R^{b-j} = \sum\limits_{i=1}^b (-1)^{i+1} \binom{b+1}{i} \xi^{b-1+j-i}\cdot R^{b+i-j}.$$
Using the fact that $R^s=0$ for $s \ge b+1$ and applying the inductive hypothesis, we deduce that
$$\xi^{b-1+j} \cdot R^{b-j} = \sum\limits_{i=1}^j (-1)^{i+1} \binom{b+1}{i} \xi^{b-1+j-i}\cdot R^{b+i-j} = \sum\limits_{i=1}^j (-1)^{i+1} \binom{b+1}{i} \binom{b+j-i}{j-i}.$$
Using the standard notation $\binom{x}{k} = \frac{x(x-1)\ldots (x-k+1)}{k!}$ for any integer $k \ge 1$, we can rewrite the above as
$$\xi^{b-1+j} \cdot R^{b-j} = \sum\limits_{i=1}^j (-1)^{i+1} \binom{b+1}{i} \binom{b+j-i}{b}=\sum\limits_{i=1}^b (-1)^{i+1} \binom{b+1}{i} \binom{b+j-i}{b}.$$
Now for any numerical polynomial $P(x)$ we have that $\sum_{j=0}^n (-1)^j\binom{n}{j} P(j)=0$ if $\deg P <n$ \cite[p.190]{GKP}. Using this we get that the last sum is $\binom{b+j}{j}$ and \eqref{monom} is proved.

Now $c_1(\E) \equiv a \xi + c R$ for some $a, c \in \ZZ$. Since $\E$ is globally generated we have that $c_1(\E)$ is effective, hence $a \ge 0$. Also, using again that $R^s=0$ for $s \ge b+1$, we get
$$0 = c_1(\E)^{2b-1} = \sum\limits_{k=1}^{2b-1} \binom{2b-1}{k} a^kc^{2b-1-k} \xi^k \cdot R^{2b-1-k} = \sum\limits_{j=0}^b \binom{2b-1}{b-1+j} a^{b-1+j}c^{b-j} \xi^{b-1+j} \cdot R^{b-j}$$
and using \eqref{monom} we get that
$$0 = \sum\limits_{j=0}^b \binom{2b-1}{b-1+j} \binom{b+j}{j} a^{b-1+j}c^{b-j} = a^{b-1} \sum\limits_{j=0}^b \binom{2b-1}{b-1+j} \binom{b+j}{j} a^jc^{b-j}.$$
An easy calculation gives that 
$$\sum\limits_{j=0}^b \binom{2b-1}{b-1+j} \binom{b+j}{j} a^jc^{b-j} = \binom{2b-1}{b}(a+c)^{b-1}(2a+c)$$
so that we get that
$$\binom{2b-1}{b}a^{b-1}(a+c)^{b-1}(2a+c)=0$$
that is either $a = 0$ or $a > 0, c = -a, -2a$. But the case $c = -2a$ cannot occur, since then $c_1(\E) \sim a\xi - \pi^*\O_{\PP^b}(2a)$. But 
$$H^0(a \xi - \pi^*\O_{\PP^b}(2a)) \cong H^0(S^a T_{\PP^b}(-2a)) \cong H^b(S^a \Omega^1_{\PP^b}(2a-b-1))^*$$
and the exact sequence
$$0 \to S^a \Omega^1_{\PP^b}(2a-b-1) \to \O_{\PP^b}(a-b-1)^{\oplus \binom{a+b}{a}} \to \O_{\PP^b}(a-b)^{\oplus \binom{a+b-1}{a-1}} \to 0$$
shows that $H^b(S^a \Omega^1_{\PP^b}(2a-b-1))=0$, contradicting the fact that $c_1(\E)$ is globally generated.

Observe that the line bundle $\xi - \pi^*\O_{\PP^b}(1)$ is globally generated and $h^0(\xi - \pi^*\O_{\PP^b}(1)) = h^0(T_{\PP^b}(-1)) = b+1$, so that it gives a morphism $\pi' : X \to \PP^b$. This is just the restriction of the second projection given by the inclusion $\PP(T_{\PP^b}) \subset \PP^b \times \PP^b$. Also note that $\pi'$ is again a $\PP^{b-1}$-bundle structure on $X$. 

In any case, we have that either $a = 0$ and $c_1(\E) = \pi^*\O_{\PP^2}(c)$ or $a > 0, c = -a$ and $c_1(\E) = (\pi')^*\O_{\PP^b}(a)$. We can now apply Lemma \ref{pullback} and deduce that there is a rank $r$ vector bundle $\H$ on $\PP^b$ such that $\E = \pi^*\H$ or $(\pi')^*\H$. In the first case $\E = \pi^*\H$, setting $\F = T_{\PP^b}, L=0$ and $\G = \H(-b-1)$ we have, by Lemma \ref{exa}, that
\begin{equation}
\label{van}
H^q(\G \otimes S^k \Omega^1_{\PP^b})=0 \ \hbox{for} \ q \ge 0, 0 \le k \le b-1.
\end{equation}
Now the Euler sequence gives, when $k \ge 1$, the exact sequence
$$0 \to \G \otimes S^k \Omega^1_{\PP^b}  \to \G(-k)^{\oplus \binom{k+b}{k}} \to \G(-k+1)^{\oplus \binom{k+b-1}{k-1}} \to 0$$
and \eqref{van} implies that $H^q(\G(-k)) = 0$ for $q \ge 0$ and $0 \le k \le b-1$. But then $\G(1)$ is an Ulrich vector bundle for $(\PP^b, \O_{\PP^b}(1))$ and it follows by \cite[Prop.~2.1]{es} (or \cite[Thm.~2.3]{b1}) that $\G(1) \cong \O_{\PP^b}^{\oplus r}$ and therefore $\E \cong (\pi^*(\O_{\PP^b}(b)))^{\oplus r}$. The proof in the other case $\E = (\pi')^*\H$ is analogous. Finally $c_1(\E) = \pi^*(\O_{\PP^b}(br))$ whence for any $1 \le s \le b$ we get that $c_1(\E)^s = b^sr^sR^s \ne 0$.
\end{proof}

\section{Blow-ups at points not on lines}

We now lay out the proof of the main theorem. 

The following result is quite standard, but we include it for lack of reference. For the second assertion see also \cite[Cor.~6]{se}.

\begin{lemma}
\label{scopp} 
Let $X \subseteq \PP^N$ be a smooth irreducible variety of dimension $n \ge 1$ and let $H$ be its hyperplane divisor. Let $x \in X$ be a point and let $\mu_x : \widetilde X \to X$ be the blow-up of $X$ at $x$ with exceptional divisor $E$. Then $\mu_x^*H-E$ is ample if and only if there is no line $L$ such that $x \in L \subseteq X$. Also $\mu_x^*(tH)-E$ is very ample for every $t \ge 2$.
\end{lemma}
\begin{proof} 
If there is a line $L \subseteq X$ passing through $x$, letting $\widetilde L$ be its strict transform, we have
$$(\mu_x^*H-E) \cdot \widetilde L = H \cdot L - 1 = 0$$
and therefore $\mu_x^*H-E$ is not ample.

Viceversa assume that $\mu_x^*H-E$ is not ample. Since $H$ is very ample, it follows that $\mu_x^*H-E$ is globally generated, hence, for every irreducible subvariety $Z$ of dimension $d \ge 1$ of $\widetilde X$ we have that and $(\mu_x^*H-E)^d \cdot Z \ge 0$ by Kleiman's theorem \cite[Thm.~1.4.9]{laz}. Now $\mu_x^*H-E$ is not ample, hence Nakai-Moishezon-Kleiman's theorem \cite[Thm.~1.2.23]{laz} implies that there is an irreducible subvariety $\widetilde V$ of dimension $d \ge 1$ of $\widetilde X$ such that and 
\begin{equation}
\label{mult1}
(\mu_x^*H-E)^d \cdot  \widetilde V = 0.
\end{equation}
Let $V = \mu_x(\widetilde V)$. It must be that $x \in V$, for otherwise \eqref{mult1} gives that $H^d \cdot V = 0$, a contradiction. Note that it cannot be that $\widetilde V \subset E$ for otherwise we get the contradiction $V=\{x\}$. Hence $\mu_x$ gives an isomorphism $\widetilde V \setminus E \cap \widetilde V \cong V \setminus \{x\}$ and therefore $\widetilde V$ is the strict transform of $V$. Therefore, by \cite[Cor.~II.7.15]{ha}, we find that $\widetilde V$ is  the blow-up of $V$ at $x$ with exceptional divisor $E_{|\widetilde V}$. By \cite[Lemma 5.1.10]{laz} we have that $E_{|\widetilde V}^d = (-1)^{d+1} {\rm mult}_x(V)$ and \eqref{mult1} gives
$$0 = (\mu_x^*H-E)^d \cdot  \widetilde V = H^d \cdot V + (-1)^d(-1)^{d+1} {\rm mult}_x(V)$$
that is 
$$H^d \cdot V = {\rm mult}_x(V).$$
Then $V \subset \PP^N$ is a variety having one point whose multiplicity is the degree of $V$ and it follows that there is a line $L \subseteq V$ passing through $x$. 

Finally the very ampleness of $\mu_x^*(tH)-E$ when $t \ge 2$ follows by \cite[Thm.~2.1]{bs1} since $\mathcal J_{\{x\}/X}(1)$ is globally generated.
\end{proof}

Using this we have the following result. The idea comes in part from \cite[Thm.~0.1]{k}.

\begin{thm}
\label{0reg}
Let $X \subseteq \PP^N$ be a smooth irreducible variety of dimension $n \ge 1$. Let $\E$ be a rank $r$ vector bundle on $X$ that is $0$-regular with respect to $\O_X(t)$ for some $t \ge 1$. Then $c_1(\E)$ has the following positivity property: For every $x \in X$ and for every subvariety $Z \subseteq X$ of dimension $k \ge 1$ passing through $x$, we have that 
\begin{equation}
\label{ineq1}
c_1(\E)^k \cdot Z \ge r^k {\rm mult}_x(Z)
\end{equation}
holds if either $t \ge 2$ or $t=1$ and $x$ does not belong to a line contained in $X$. 

In particular $c_1(\E)$ is ample if either $t \ge 2$ or $t=1$ and there is no line contained in $X$.

Moreover assume that either $t \ge 2$ or $t=1$ and $X$ is not covered by lines. Then $\E$ is big and
\begin{equation}
\label{ineq1'} c_1(\E)^n \ge r^n
\end{equation}
\end{thm}
\begin{proof}
Let $H$ be the hyperplane divisor of $X$. Let $x \in X$ be a point and assume that either $t \ge 2$ or $t=1$ and $x$ does not belong to a line contained in $X$. Let $\mu : \widetilde X \to X$ be the blow-up of $X$ at $x$ with exceptional divisor $E$. Since $H$ is very ample, it follows that $\widetilde H := \mu^*(tH)-E$ is globally generated and it is also ample by Lemma \ref{scopp}. As is well-known (see for example \cite[Proof of Lemma 1.4]{bel}), for every integer $s$ such that $0 \le s \le n-1$ we have that
\begin{equation}
\label{rj2}
R^j \mu_* \O_{\widetilde X}(sE) =  \begin{cases} \O_X & {\rm if} \ j=0 \\ 0 & {\rm if} \  j>0  \end{cases}.
\end{equation}
Then, for every $i>0$, \eqref{rj2} and the Leray spectral sequence give
$$H^i(\widetilde X, (\mu^*\E)(-E)(-i\widetilde H)) = H^i(\widetilde X,\mu^*(\E(-it))((i-1)E)) \cong H^i(X,\E(-it))=0$$
since $\E$ is $0$-regular with respect to $\O_X(t)$. Therefore $(\mu^*\E)(-E)$ is $0$-regular with respect to $\widetilde H$ and it follows by \cite[Thm.~1.8.5]{laz} that $(\mu^*\E)(-E)$ is globally generated. But then also $c_1(\mu^*\E(-E)) = \mu^*(c_1(\E))-rE$ is globally generated, whence nef. Now let $Z \subseteq X$ be a subvariety of dimension $k \ge 1$ passing through $x$ and let $\widetilde Z$ be its strict transform on $\widetilde X$. By Kleiman's theorem \cite[Thm.~1.4.9]{laz} we deduce that
\begin{equation}
\label{pri}
(\mu^*(c_1(\E))-rE)^k \cdot \widetilde Z \ge 0.
\end{equation}
Now, using \cite[Lemma 5.1.10]{laz}, we see that $(E_{|\widetilde Z})^k = (-1)^{k+1} {\rm mult}_x(Z)$ and \eqref{pri} implies that
$$0 \le (\mu^*(c_1(\E))-rE)^k \cdot \widetilde Z = c_1(\E)^k \cdot Z + (-1)^kr^k(-1)^{k+1} {\rm mult}_x(V) = c_1(\E)^k \cdot Z - r^k {\rm mult}_x(Z).$$
This proves \eqref{ineq1}. Now if either $t \ge 2$ or $t=1$ and there is no line contained in $X$, then $c_1(\E)$ is ample by Nakai-Moishezon-Kleiman's criterion \cite[Thm.~1.2.23]{laz}.

Assume that either $t \ge 2$ or $t=1$ and $X$ is not covered by lines. Choosing $Z=X$ we get that
$$c_1(\E)^n \ge r^n.$$
Moreover by \cite[Thm.~2.5]{dps} we have, considering the partition $(1,\ldots,1)$ of $n$ and the corresponding Schur polynomial \cite[\S 8.3]{laz}, that 
\begin{equation}
\label{seg}
s_{(1,\ldots,1)}((\mu^*\E)(-E)) \ge 0.
\end{equation}
Denoting with $s_n$ the $n$-th Segre class, we know by \cite[Exa.~8.3.5]{laz} that 
$$s_{(1,\ldots,1)}((\mu^*\E)(-E))=s_n(((\mu^*\E)(-E))^*)$$ 
hence we deduce from \eqref{seg} that
$$(-1)^ns_n((\mu^*\E)(-E)) \ge 0$$
and therefore
$$(-1)^n \sum\limits_{j=0}^n \binom{r-1+n}{r-1+j} s_j(\mu^*\E)E^{n-j} \ge 0$$
that is
$$s_n(\E^*) = (-1)^n s_n(\E) \ge \binom{r-1+n}{r-1} > 0$$
and then $\O_{\PP(\E)}(1)$ is big, namely $\E$ is big.
\end{proof}

\begin{remark}
\label{altrapos}
With the same proof, one can show that if $\E$ is a rank $r$ vector bundle $0$-regular with respect to $\O_X(1)$, $x \in X$ does not belong to a line contained in $X$, $k \ge 1$ and $Z \subseteq X$ is a $k$-dimensional subvariety passing through $x$, then 
$$s_k(\E^*) \cdot Z \ge \binom{r-1+k}{r-1} {\rm mult}_x(Z).$$
Moreover, if $X$ is not covered by lines, also $s_{\lambda}((\mu^*\E)(-E)) \ge 0$ for any degree $n$ Schur polynomial associated to a partition $\lambda$ \cite[\S 8.3]{laz}. This in turn translates into positivity of polynomials in Chern classes of $\E$. For example if $\lambda=(n,0,\ldots,0)$ then one gets that $c_n(\E) \ge \binom{r}{n}$ when $r \ge n$.
\end{remark}

Now Theorem \ref{main} follows.

\begin{proof}[Proof of Theorem \ref{main}]

\hskip 3cm

Let $\E$ be a rank $r$ Ulrich vector bundle for $(X,\O_X(t))$. Then $\E$ is $0$-regular with respect to $\O_X(t)$, so that we can apply Theorem \ref{0reg} and conclude that \eqref{ineq} holds and that $c_1(\E)$ is ample if either $t \ge 2$ or $t=1$ and there is no line contained in $X$. Now assume that $t \ge 2$ or $t=1$ and $X$ is not covered by lines. Again Theorem \ref{0reg} implies that \eqref{ineq'} holds and that $\E$ is big. Moreover, if $r \ge 2$ we can apply \cite[Thm.~1]{si} since $\E$ is globally generated. Note that it cannot be that $\E \cong \O_X^{\oplus (r-1)} \oplus \det \E$, for otherwise $\O_X$ would be an Ulrich line bundle for $(X,\O_X(t))$. But then Lemma \ref{c1=0} would imply that $(X,\O_X(t))=(\PP^n, \O_{\PP^n}(1))$, hence $t=1$ and $X$ is covered by lines, a  contradiction. Therefore \cite[Thm.~1]{si} gives that $c_1(\E)^n \ge h^0(\E) - r = r(d-1)$ by \cite[(3.1)]{b1}. Thus also \eqref{ineq''} holds.
\end{proof}

\section{The case of surfaces}

We are now ready to prove our classification in the case of surfaces.

\begin{proof}[Proof of Theorem \ref{main2}]

\hskip 3cm

If $(S, \O_S(1), \E)$ is as in (i) or (ii), it follows by \cite[Prop.~2.1]{es} (or \cite[Thm.~2.3]{b1}) and Lemma \ref{exa}(i) that $\E$ is a rank $r$ Ulrich vector bundle on $S$ with $c_1(\E)^2=0$.

Vice versa let $\E$ be a rank $r$ Ulrich vector bundle on $S$ such that $c_1(\E)^2=0$. 

If $(S, \O_S(1)) = (\PP^2,\O_{\PP^2}(1))$ it follows by \cite[Prop.~2.1]{es} (or \cite[Thm.~2.3]{b1}) that $\E = \O_{\PP^2}^{\oplus r}$, so that we are in case (i).

Assume from now on that $(S, \O_S(1)) \ne (\PP^2,\O_{\PP^2}(1))$. We will prove that $(S,\O_S(1), \E)$ is as in (ii).

By Remark \ref{c1=0} we have that $c_1(\E) \ne 0$. We now claim that $c_2(\E)=0$.

Since $\E$ is $0$-regular with respect to $\O_S(1)$, it is globally generated by \cite[Thm.~1.8.5]{laz}. It follows by \cite[Thm.~2.5]{dps} that $0 \le c_2(\E) \le c_1(\E)^2=0$, that is $c_2(\E)=0$.

Alternatively, to show that $c_2(\E)=0$, consider the map
$$\lambda_{\E} : \Lambda^r H^0(\E) \to H^0(\det \E).$$
As shown in the proof of Lemma \ref{pullback}, the linear system $|\Im \lambda_{\E}|$ is base-point free. Choose now a general subspace $V \subseteq H^0(\E)$ of dimension $r$, so that we get a morphism $\varphi : V \otimes \O_S \to \E$ whose degeneracy locus $D = D_{r-1}(\varphi)$ is a smooth curve by \cite[Statement(folklore), \S 4.1]{ba}. Since $\varphi$ is generically an isomorphism we get, as in \cite[Proof of Prop.~1.1]{gl}, an exact sequence of sheaves
\begin{equation}
\label{l1}
0 \to V \otimes \O_S \to \E \to \L \to 0
\end{equation}
and its dual
\begin{equation}
\label{l2}
0 \to \E^* \to V^* \otimes \O_S \to \L' \to 0
\end{equation}
where $\L$ and $\L'$ are two line bundles on $D$ with $\L \cong N_{D/S} \otimes (\L')^{-1}$. Note that 
$$c_2(\E) = \deg \L'.$$
Now both $\L$ and $\L'$ are globally generated by \eqref{l1} and \eqref{l2}. Since $D^2=0$ we have that $D = \Gamma_1 \cup \ldots \cup \Gamma_s$ with all $\Gamma_i$ smooth irreducible, $\Gamma_i \cap \Gamma_j = \emptyset$ for $i \neq j$ and $\Gamma_i \equiv \Gamma_j$, so that $D \equiv s \Gamma_i$ for all $i$ and therefore $\Gamma_i^2=0$ for all $i$. Since $\O_S(D)$ is globally generated we get that also $\O_S(\Gamma_i)_{|\Gamma_i} \cong \O_S(D)_{|\Gamma_i}$ is globally generated and has degree $0$, whence $\O_S(\Gamma_i)_{|\Gamma_i} \cong \O_{\Gamma_i}$. Also
$$\L_{|\Gamma_i} \cong (N_{D/S} \otimes (\L')^{-1})_{|\Gamma_i} \cong (\L'_{|\Gamma_i})^{-1}$$
and both $\L_{|\Gamma_i}$ and $\L'_{|\Gamma_i}$ are globally generated, whence $\L_{|\Gamma_i} \cong \L'_{|\Gamma_i} \cong \O_{\Gamma_i}$. Therefore
$$\L \cong \L' \cong \bigoplus\limits_{i=0}^s \O_{\Gamma_i}$$
and this gives that $c_2(\E) = \deg \L' = 0$.

Thus the claim that $c_2(\E) = 0$ is proved and we proceed with the proof.

Let $H$ be a hyperplane divisor on $S$. 

Since $c_1(\E)^2=0$ we have that \eqref{ineq'} does not hold, hence Theorem \ref{main} implies that $S \subseteq \PP^N = \PP H^0(H)$ is covered by lines. By \cite[Thm.~1.4]{lp} it follows that $(S,H)$ is a linear $\PP^1$-bundle over a smooth curve (possibly with two $\PP^1$-bundle structures in the case of $\PP^1 \times \PP^1$). Hence there is a smooth curve $B$, a divisor $L$ on $B$, a rank $2$ vector bundle $\F$ on $B$ (that we can assume to be normalized as in \cite[V.2.8.1]{ha}) such that $S \cong \PP(\F)$, $H = \xi + \pi^*L$, where $\pi: S \to B$ is the projection and $\xi$ is the tautological line bundle on $\PP(\F)$. 

We claim that there is a $\PP^1$-bundle structure $p : S \to C$ onto a smooth curve $C$ and a line bundle $M$ on $C$ such that $\det \E = p^*M$. As a matter of fact we will have that either $p =\pi$ and $C=B$ or $S \cong \PP^1 \times \PP^1, C = \PP^1$ and $p : S \to \PP^1$ is the other projection (besides $\pi$).  

To this end observe that we can write $\det \E = a \xi + \pi^*M$ for some $a \in \ZZ$ and some line bundle $M$ on $B$. Then there is $b \in \ZZ$ such that $\det \E \equiv a \xi + b f$, where $f$ is a fiber of $\pi$. Set $e = -\deg \F, l = \deg L$ and $g = g(B)$, so that $H \equiv \xi +lf$ and $K_S \equiv -2\xi + (2g-2-e)f$. Since $(\det \E)^2=0$ we get that $-a^2e + 2ab = 0$ and therefore either $a=0$ or $b = \frac{ae}{2}$. But $\det \E$ is globally generated, hence $a= \det \E \cdot f \ge 0$ and if $a > 0$ then $- \frac{ae}{2} = \det \E \cdot \xi \ge 0$, that is $e \le 0$. Moreover observe that, since $c_1(\E) \ne 0$, we have that $H \cdot \det \E > 0$, and we deduce that either 
\begin{equation}
\label{caso1}
a=0 \ \hbox{and} \ b > 0
\end{equation}
or 
\begin{equation}
\label{caso2}
a > 0, e \le 0, b = \frac{ae}{2} \ \hbox{and} \ l > \frac{e}{2}.
\end{equation}
Suppose we are in case \eqref{caso2}. By \cite[Prop.~2.2]{c1} we have that
\begin{equation}
\label{uno}
c_1(\E) \cdot H = \frac{r}{2}(3H^2+H \cdot K_S) \ \hbox{and} \ c_2(\E) = \frac{1}{2}(c_1(\E)^2-c_1(\E) \cdot K_S) - rH^2 + r \chi(\O_S).
\end{equation}
Since $H^2=2l-e$ and $H \cdot K_S = 2g-2+e-2l$, using $c_1(\E)^2=c_2(\E)=0$ and replacing in \eqref{uno} we get
$$a(l - \frac{e}{2}) = r(2l-e+g-1) \ \hbox{and} \ a(1-g) = r(2l-e+g-1).$$
These give
$$a(l - \frac{e}{2}+g-1)=0$$
that is $l=\frac{e}{2}-g+1$. Since $l > \frac{e}{2}$ we get that $g=0$. As is well-known \cite[Cor.~V.2.13]{ha}, when $g=0$, we have that $e \ge 0$. Therefore $e=0, l = 1$ and $S \cong \PP^1 \times \PP^1$.

We conclude that either we are in case \eqref{caso1} and $\det \E \cong \pi^*M$ for some line bundle $M$ on $B$ (in this case we set $C=B$ and $p=\pi$) or we are in case \eqref{caso2} and then $B \cong \PP^1$ and $\det \E \cong p^*(\O_{\PP^1}(a))$ where $p : S \to \PP^1$ is the other projection besides $\pi$ (in this case we set $C=\PP^1, L = \O_{\PP^1}(1)$ and $M = \O_{\PP^1}(a)$).

Thus the claim is proved and Lemma \ref{pullback} implies that there is a rank $r$ vector bundle $\H$ on $C$ such that $\E \cong p^*\H$. Set $\G = \H(-L-\det \F)$, so that $\E \cong p^*(\G(L+ \det \F))$. Applying Lemma \ref{exa} to $p$ we deduce that $H^q(\G(-L))=0$ for $q \ge 0$.

Thus $(S,\O_S(1), \E)$ is as in (ii).
\end{proof}

\section{The case of threefolds}

In both proofs we will use a result of Lanteri and Palleschi together with a result of Fujita.

\subsection{A quick guide to threefolds covered by lines}
\label{lapa}

\hskip 3cm

Let $X \subseteq \PP^N$ be a smooth irreducible threefold covered by lines, let $x \in X$ be a general point and let $\H_x$ be the Hilbert scheme of lines contained in $X$ and passing through $x$. All the consequences listed below follow from \cite[Thm.~1.4]{lp} and \cite[Introduction]{fu}. To simplify the list, following our purposes, we also assume that $X$ has Picard rank $\rho(X) \ge 2$. 

We have that $\dim \H_x \le 1$ with equality if and only if 
\begin{itemize}
\item [(1)] $(X, \O_X(1))$ is a linear $\PP^2$-bundle over a smooth curve.
\end{itemize}
Assume that $(X, \O_X(1))$ is not as in (1). Then $(X, \O_X(1))$ is
either one of the following
\begin{itemize}
\item [(2)] $(\PP^1 \times \PP^1 \times \PP^1, \O_{\PP^1}(1) \boxtimes \O_{\PP^1}(1) \boxtimes\O_{\PP^1}(1))$;
\item [(3)] $(\PP(T_{\PP^2}), \O_{\PP(T_{\PP^2})}(1))$;
\item [(4)] a linear quadric fibration over a smooth curve;
\item [(5)] a linear $\PP^1$-bundle over a smooth surface.
\end{itemize}
Note that the cases (2)-(3) above are listed as (e) in \cite[Thm.~1.4]{lp}, while all the other cases in \cite[Thm.~1.4(e)]{lp} have $\rho(X) = 1$, hence they do not appear in our list, except for the Del Pezzo threefold of degree $7$, that is also of type (5).

\subsection{Proof of Theorems}

\begin{proof}[Proof of Theorem \ref{main3}]

\hskip 3cm
 
If $(X, \O_X(1), \E)$ is as in (i), (ii) or (iii), it follows by \cite[Prop.~2.1]{es} (or \cite[Thm.~2.3]{b1}) and Lemma \ref{exa} that $\E$ is a rank $r$ Ulrich vector bundle on $S$ with $c_1(\E)^2=0$.

Now suppose that $\E$ is a rank $r$ Ulrich vector bundle on $X$ with $c_1(\E)^2=0$. 

If $(X,\O_X(1)) = (\PP^3,\O_{\PP^3}(1))$ it follows by \cite[Prop.~2.1]{es} (or \cite[Thm.~2.3]{b1}) that $\E = \O_{\PP^3}^{\oplus r}$, so that we are in case (i).

Assume from now on that $(X,\O_X(1)) \ne (\PP^3,\O_{\PP^3}(1))$. We will prove that $(X, \O_X(1), \E)$ is as in (ii) or (iii).

By Remark \ref{c1=0} we have that $c_1(\E) \ne 0$. Moreover by Lemma \ref{pic} we have that $\rho(X) \ge 2$.

Since $c_1(\E)^3=0$ we have that \eqref{ineq'} does not hold, hence Theorem \ref{main} implies  that $X$ is covered by lines. Therefore $(X,\O_X(1))$ is as in (1)-(5) in \ref{lapa}.
 
In case (1) we know by Lemma \ref{cs} that $(X,\O_X(1),\E)$ is as in (ii). In cases (2) and (3) we know by Propositions \ref{qf} and \ref{exa1} that $c_1(\E)^2 \ne 0$, hence these cases do not occur.

As case (4) is excluded by hypothesis, it remains to study case (5) with $\Delta(\F) \ne 0$. Hence we have a $\PP^1$-bundle structure $\pi : X \to B$. We let $\xi$ be the tautological line bundle.

We claim that $\det \E = \pi^*M$ for some line bundle $M$ on $B$.

We can write $\det \E = a \xi + \pi^*M$ for some $a \in \ZZ$ and some line bundle $M$ on $B$.

Assume that $a \ne 0$. We have 
$$0 = c_1(\E)^2 = a^2 \xi^2 + 2a \xi \pi^*M + (\pi^*M)^2.$$
Since $\F$ has rank $2$ we know that $\xi^2 = \xi \pi^*c_1(\F) - \pi^*c_2(\F)$ and we get
$$(a^2 \pi^*c_1(\F) + 2a \pi^*M) \xi + \pi^*(M^2) - a^2 \pi^*c_2(\F) = 0$$
that is
\begin{equation}
\label{equaz}
a \pi^*(ac_1(\F) + 2M) \xi + \pi^*(M^2 - a^2 c_2(\F)) = 0.
\end{equation}
If $ac_1(\F) + 2M \not\equiv 0$, we can find an ample divisor $A$ on $B$ such that $A \cdot (ac_1(\F) + 2M) \ne 0$. Intersecting in \eqref{equaz} with $\pi^*A$
we get the contradiction
$$a A \cdot (ac_1(\F) + 2M) = 0.$$
Therefore $ac_1(\F) + 2M \equiv 0$ and \eqref{equaz} gives that $M^2 = a^2 c_2(\F)$. But then
$$a^2 \Delta(\F)=4a^2c_2(\F)-a^2c_1(\F)^2=0$$
again a contradiction. Therefore $a=0$ and then $\det \E = \pi^*M$.

Thus the claim is proved and Lemma \ref{pullback} implies that there is a rank $r$ vector bundle $\H$ on $B$ such that $\E \cong \pi^*\H$. Setting $\G = \H(-L-\det \F)$, we find that $\E \cong \pi^*(\G(L+ \det \F))$ and Lemma \ref{exa} implies that $H^q(\G(-L))=H^q(\G(-2L) \otimes \F^*)=0$ for $q \ge 0$. Hence we are in case (iii).
\end{proof}

\begin{proof}[Proof of Theorem \ref{main4}]

\hskip 3cm

If $(X, \O_X(1), \E)$ is as in (i) or (ii), it follows by Propositions \ref{qf} and \ref{exa1} that $\E$ is a rank $r$ Ulrich vector bundle on $X$ with $c_1(\E)^3=0$ and $c_1(\E)^2 \ne 0$. 
 
Now suppose that $\E$ is a rank $r$ Ulrich vector bundle on $X$ with $c_1(\E)^3=0$ and $c_1(\E)^2 \ne 0$. 

Since $c_1(\E)^3=0$ we have by Lemma \ref{pic} that $\rho(X) \ge 2$. Also \eqref{ineq'} does not hold, whence Theorem \ref{main} implies that $X$ is covered by lines.  Therefore $(X,\O_X(1))$ is as in (1)-(5) in \ref{lapa}.

In cases (2) and (3) we can apply Propositions \ref{qf} and \ref{exa1} and conclude that $(X, \O_X(1), \E)$ is as in (i) or (ii). 

Cases (4) and (5) are excluded by hypothesis unless $(X,\O_X(1))$ is also a Del Pezzo threefold, thus belonging to cases (2) or (3).

In the remaining case (1) we have a $\PP^2$-bundle structure $\pi : X \cong \PP(\F) \to B$ and $\deg \F < 0$ by hypothesis.  We let $\xi$ be the tautological line bundle. We will first prove that $\det \E = \pi^*M$ for some line bundle $M$ on $B$.

We can write 
$$\det \E = a \xi + \pi^*M$$ 
for some $a \in \ZZ$ and some line bundle $M$ on $B$. Then there is $b \in \ZZ$ such that $\det \E \equiv a \xi + b f$, where $f$ is a fiber of $\pi$. Set $e = -\deg \F > 0$ and suppose that $a \ne 0$. We have
$$0 = c_1(\E)^3 = a^2(-ae + 3b)$$
and therefore $b = \frac{ae}{3}$ and 
$$c_1(\E) = a(\xi + \frac{e}{3}f).$$ 
Since $\F$ is normalized we have that $\xi$ is effective, hence
$$0 \le c_1(\E)^2 \cdot \xi = - a^2 \frac{e}{3} < 0$$
a contradiction.

Therefore $a=0$, that is $\det \E = \pi^*M$. By Lemma \ref{pullback} we get that there is a rank $r$ vector bundle $\H$ on $B$ such that $\E \cong \pi^*\H$. But then $c_1(\E)^2 = 0$, so that this case does not occur.  
\end{proof}

\end{document}